\documentclass[twoside,12pt]{article}
\usepackage{amsmath,amsthm,amssymb,amscd}
\usepackage{amssymb}
\usepackage{amsthm}
\usepackage{amsmath}
\usepackage{latexsym}
\usepackage{epsfig}
\usepackage{amscd}
\usepackage{graphicx,wrapfig}
\usepackage{subcaption}

\numberwithin{equation}{section}

\title {\Large{On the subgroup generated by solutions of Pell's equation}}

\author{Elena C. Covill, Mohammad Javaheri, Nikolai A. Krylov\\ ~ \\
Siena College, Department of Mathematics\\
515 Loudon Road, Loudonville NY 12211\\ ~ \\
ec20covi@siena.edu, mjavaheri@siena.edu, nkrylov@siena.edu
}

\date {}

\begin{document}
\newtheorem{theorem}{Theorem}
\newtheorem{lemma}[theorem]{Lemma}
\newtheorem{cor}[theorem]{Corollary}
\newtheorem{conj}[theorem]{Conjecture}
\newtheorem{prop}[theorem]{Proposition}
\newtheorem{question}{Problem}
\theoremstyle{definition}
\newtheorem{definition}[theorem]{Definition}
\newtheorem{example}[theorem]{Example}
\newtheorem{xca}[theorem]{Exercise}
\theoremstyle{remark}
\newtheorem{remark}[theorem]{Remark}
\numberwithin{equation}{section}

\def\natu           {\mathbb N}
\def\inte 		{\mathbb Z}
\def\real		{\mathbb R}
\def\rati		{\mathbb Q}
\def\P 		{\cal P}
\def\O 		{{\cal O}}
\def\S		{{\cal S}}

\def\lla		{\longleftarrow}
\def\lra		{\longrightarrow}
\def\ra		{\rightarrow}
\def\Ra		{\Rightarrow}
\def\hra		{\hookrightarrow}
\def\lmt		{\longmapsto}
\def\lam            {\lambda}
\def\del		{\delta}
\def\veps		{\varepsilon}
\def\re		{\rangle}
\def\le		{\langle}

\maketitle

\begin{abstract}
Equivalence classes of solutions of the
Diophantine equation $a^2+mb^2=c^2$ form an infinitely generated abelian group $G_m$, where $m$ is a fixed square-free positive integer. Solutions of Pell's equation 
$x^2-my^2=1$ generate a subgroup $P_m$ of $G_m$. We prove that $P_m$ and $G_m/P_m$ have infinite rank for all $m>1$. We also give several examples of $m$ for which $G_m/P_m$ has nontrivial torsion.
\end{abstract}

\noindent {\bf Keywords}: Pell's Equation, Reciprocity Laws, Continued Fractions, Ideal Class Group.
\\
{\bf 2010 Mathematics Subject Classification}: Primary 11D09, Secondary 11R11, 11A55, 11R20.

\section{Introduction}
 Solutions $(a_1,b_1,c_1)$ and $(a_2,b_2,c_2)$ of the Diophantine equation
\begin{equation}
\label{main} a^2 + mb^2 = c^2
\end{equation}
produce another solution under the binary operation:
\begin{equation}
\label{operation} (a_1,b_1,c_1)*(a_2,b_2,c_2)=(a_1a_2 - mb_1b_2, a_1b_2 + a_2b_1,
c_1c_2).
\end{equation}
The operation (1.2) induces an abelian group structure on $G_m$, the set of equivalence classes of solutions of \eqref{main}, as made precise in the following definition (for an alternative definition in terms of  equivalence classes of 
ordered triples of integers, see \cite{Baldisserri, Kulzer}).
\begin{definition}
Let $
F_m=\rati[\sqrt{-m}]= \left \{a+b\sqrt{-m}:~a,b\in \rati \right \}
$
be the quadratic field associated with a square-free positive integer $m$. Let $\S_m$ denote the multiplicative subgroup of $F_m \backslash \{0\}$ consisting of all nonzero elements such that $a^2+mb^2$ is a square of a rational number. We let
$
G_m= \S_m /\rati^*.
$
\end{definition} 
An equivalence class $[a+b\sqrt{-m}]\in G_m$ can 
be represented by a \emph{primitive} triple $[x,y,z]$ with $x^2 + my^2 = z^2$ (a triple $[x,y,z]$ is primitive if 
$(x,y,z)\in\inte\times\inte\times\natu$ and $\gcd(x,y,z)=1$). This representation is unique up to the equivalence $[x,y,z] \sim [-x,-y,z]$. The group operation on $G_m$ induced by the operation in \eqref{operation} can be written as
 \begin{equation}\label{go}
[x,y,z] + [a,b,c] =\left  [\dfrac{xa - myb}{g}, \dfrac{xb+ya}{g}, \dfrac{zc}{g} \right],
\end{equation}
for
$[x,y,z],[a,b,c]\in G_m$ and $g=\gcd(xa-myb,xb+ya,zc)$. 

For $m>1$, the group $G_m$ has been studied by various authors \cite{Baldisserri,Krylov,Kulzer,Lemmermeyer}, who have shown, 
among other results, that $G_m$ is infinitely generated and has nontrivial torsion $\inte/3\inte$ only when $m=3$. 

In this paper, we study the subgroup of $G_m$ generated by the solutions of Pell's equation
\begin{equation}\label{pelleq}
X^2-mY^2=1.
\end{equation}
Let $P_m \subseteq G_m$ denote the subgroup generated by all primitive triples $[1,Y,X]$ such that $(1,Y,X)$ is a solution of \eqref{main}, or equivalently $(X,Y)$ is a solution of \eqref{pelleq}. In Section 2, we show that $P_m$ has infinite rank for all square-free $m>1$ (Proposition \ref{rankpm}). We are interested in determining the rank and torsion of $G_m/P_m$. In Section 3, we prove that $G_m/P_m$ has infinite rank for all square-free $m>1$ (Proposition \ref{quotientrank}). In Section 4, several examples of quotient groups $G_m/P_m$ with nontrivial $2$-torsion are given (Table \ref{table2}) by constructing a homomorphism from $G_m$ to the ideal class group of the imaginary quadratic field $\rati[\sqrt{-m}]$.

\begin{center}
{\bf Acknowledgement}
\end{center}
 
The authors gratefully acknowledge support from the Siena Summer Scholars program that 
funds scholarly activities in which faculty members and students of Siena College collaborate 
during the summer. This article is the result of such collaboration. We also thank Michael Stoll 
for useful comments on an earlier version of this paper and for suggesting a different approach 
to the proof that $G_m/P_m$ has infinite rank.


\section{Preliminary results}

Here and throughout this paper, $m$ is a square-free integer greater than 1. 
The following lemma determines when primitivity of triples is preserved under the group operation \eqref{go} in $G_m$. 

\begin{lemma}\label{primitive}
Let $[a,b,c]$ and $[x,y,z]$ be primitive triples in $G_m$, and $p$ be a prime.
\begin{itemize}
\item[i)] If $p \nmid \gcd(c,z)$, then $p\nmid \gcd(xa - myb, xb+ya, zc)$.
\item[ii)] If $\gcd(xa - myb, xb+ya, zc)=g>1$, then $g=d^2 2^{\epsilon}$, where $d\in \natu$ and $\epsilon\in\{0,1\}$. 
\end{itemize}
\end{lemma}
\begin{proof}

i) On the contrary, suppose a prime $p$ is a common factor of $xa - myb$, $xb+ya$, and $zc$. 
Then there exist nonzero integers $s$ and $t$ such that $ps = xa-myb$ and 
$pt=xb+ya$. It follows that
\begin{eqnarray} \label{eqd1}
p(ta-sb) &=& yc^2; \\ \label{eqd2}
p(tx - sy) &=& bz^2.
\end{eqnarray}
If $p\mid y$ then $p\nmid z$ (since $[x,y,z]$ is primitive), and so $p\mid b$ by \eqref{eqd2}, which implies that $p\nmid c$ (since $[a,b,c]$ is primitive). This is a contradiction, since $p\mid zc$. Therefore $p \nmid y$. Similarly, $p\nmid b$. But then $p\mid\gcd(z,c)$ by \eqref{eqd1} and \eqref{eqd2}, which is a 
contradiction.

ii) Let $p$ be a prime and $k\geq 0$ such that $p^{2k+1} \mid g$ and $p^{2k+2}\nmid g$. We need to show that $p=2$. One writes $p^{2k+1}s = xa-myb$ and $p^{2k+1}t=xb+ya$ for some nonzero integers $s$ and
$t$ at least one of which relatively prime with $p$. Hence
\begin{equation}
\label{prop1}
p^{2k+1}(ta-sb) = yc^2~~~\text{and}~~~p^{2k+1} (tx-sy)= bz^2,
\end{equation}
and as in the proof of (i), we must have $p\nmid y$, $p \nmid b$, $p^{2k+1}\mid c^2$, and 
$p^{2k+1}\mid z^2$. Therefore, $z= p^uz_1,~c=p^vc_1$, where $u,v\in\natu$ and $\gcd(z_1,p)=\gcd(c_1,p)=1$. Then we 
can rewrite formulas (\ref{prop1}) in the forms
\begin{equation}
\label{prop1b}
p^{2k+1}(ta-sb) = yp^{2v}c_1^2~~~\text{and}~~~p^{2k+1} (tx-sy)= bp^{2u}z_1^2.
\end{equation}
By comparing the exponents of $p$ on both sides of the equations in \eqref{prop1b}, we conclude that $p\mid (ta-sb)$ and $p\mid (tx-sy)$, hence
$p\mid s(xb-ya)$ and $p\mid t(xb-ya)$. Since either $\gcd(p,t)=1$ or $\gcd(p,s)=1$, we deduce that $p\mid xb-ya$. Since $p\mid xb+ya$, we must have $p\mid 2ya$, and so $p=2$. 
\end{proof}

The general positive integer solution of Pell's equation is given by $(x,y)=(a_n,b_n)$, where 
\begin{equation}\label{genpell}
 a_n + b_n \sqrt{m}=(a_1+ b_1\sqrt{m})^n,
\end{equation}
with $(a_1,b_1)$ being the least positive solution; \cite[Theorem 7.26]{Niven}. Since $a_n^2-mb_n^2=1$, we have $\gcd(a_n,b_n)=1$ for all $n\in \mathbb{N}$.
The least positive solution $(a_1,b_1)$ is found via the periodic continued fraction expansion of $\sqrt{m}$. Let $\eta = [u_0;u_1,u_2,\ldots ]$ denote the infinite continued fraction
$$
\eta = u_0 + \dfrac{1}{u_1+\dfrac{1}{u_2 +\cdots        } }.
$$
The simple continued 
fraction expansion of $\sqrt{m}$ has the form
$$
\sqrt{m} =  [u_0;u_1,\ldots,u_r,u_1,\ldots,u_r,u_1,\ldots]  = [u_0;\overline{u_1,\ldots,u_r}],
$$
where $r$ denotes the length of the shortest period in the expansion of $\sqrt{m}$; \cite[Theorem 7.21]{Niven}.
The $i$th convergent of $[u_0;u_1,u_2,\ldots ]$ is the truncated continued fraction $h_i/k_i=[u_0;u_1,\ldots, u_i ]$ with $\gcd(h_i,k_i)=1$. All positive solutions of \eqref{pelleq} are of the form $(h_n, k_n)$. Let $l=r$ if $r$ is even and $l=2r$ if $r$ is odd. Then, one has $(a_n,b_n)=( h_{nl -1}, k_{nl-1})$. 
\\
\\
{\bf Notation.} From now on, $m>1$ is a square-free integer and $a,b$ are such that $(a,b)=(a_1,b_1)$ is the least positive solution of the Pell's equation \eqref{pelleq}.

\begin{prop}\label{rankpm}
The rank of $P_m$ is infinite.
\end{prop}

\begin{proof}
It follows from \eqref{genpell} that 
$b_{2^n} = 2a_{2^{n-1}}b_{2^{n-1}}$, and so $a_{2^{n-1}} \mid b_{2^n}$ and 
$b_{2^{n-1}} \mid b_{2^n}$ for all $n\in \mathbb{N}$. A simple finite induction 
shows that $a_{2^j} \mid b_{2^i}$ for all $i>j\geq 0$. On the other hand $\gcd(a_{2^i}, b_{2^i})=1$, and so $\gcd(a_{2^i},a_{2^j})=1$ for all $i>j \geq 0$. 
We claim that the triples $[1,b_{2^n},a_{2^n}], ~n\in\natu$, are linearly independent over $\mathbb{Z}$. To see this, suppose that there exists a finite set $\Omega \subseteq \mathbb{N}$ and nonzero integers $k_n$, $n\in \Omega$, such that the primitive representation of
$$\sum_{n\in \Omega} k_n [1,b_{2^n},a_{2^n}]$$ 
is the identity element $[1,0,1]$. For each $n\in \Omega$, let $[x_n,y_n,z_n]$ be the primitive representation of $k_n[1,b_{2^n},a_{2^n}]$. Then $z_n \mid (a_{2^n})^{|k_n|}$. Moreover, $z_n\neq 1$, since otherwise $[1,b_{2^n},a_{2^n}]$ would be an element of finite order which is possible only if $m=3$ and $[1,b_{2^n},a_{2^n}]=[1,1,2]$ (see \cite{Kulzer}), which we can remove and still have infinitely many triples. Since $\gcd(z_i,z_j)=1$ for all $i\neq j$ in $\Omega$, it follows from Lemma \ref{primitive} that the third component in the primitive representation of $\sum_{n\in \Omega}[x_n,y_n,z_n]$ is $\prod_{n \in \Omega}z_n\neq 1$. Therefore, the triples $[1,b_{2^n},a_{2^n}]$, $n\in \mathbb{N}$, are linearly independent over $\mathbb{Z}$, and so the rank of $P_m$ is infinite. 
\end{proof}

\subsection{The recursion $A_{n+1}=2aA_n-A_{n-1}$}
The sequence $a_n$, $n \in \natu$, defined by \eqref{genpell} satisfies the recursion:
\begin{equation}\label{rec}
A_{n+1}=2aA_n-A_{n-1},~n\geq 2,
\end{equation}
The following proposition follows from straightforward calculations, and therefore its proof is omitted (for example, see \cite[Ch. 2]{Elaydi}). 
\begin{prop}\label{properties}
Let $F_n,G_n$ be the solutions of the recursion \eqref{rec} with 
$F_0=0,~F_1=1$, and $G_0=2,~G_1=2a$. Let $r,s$ be distinct solutions of the quadratic equation \begin{equation}\label{equad}
x^2-2ax+1=0.
\end{equation} 
Then we have:
\begin{itemize}
\item[i)] $F_n=(r^n-s^n)/(r-s)$ and $G_n=r^n+s^n$. 
\item[ii)] $G_n=2a_n$.
\item[iii)]$F_{2n}=F_nG_n$. 
\item[iv)] $G_n^2-4(a^2-1)F_n^2=4$. 
\item[v)] $
\begin{pmatrix}
F_{n+1} & -F_n\\
F_n & - F_{n-1}
\end{pmatrix} = 
\begin{pmatrix}
2a & -1\\
1 & 0
\end{pmatrix}^n
$.
\end{itemize}
All identities hold modulo an odd prime $p$ if $r,s$ are distinct solutions of \eqref{equad} modulo $p$. 
\end{prop}

In Section 3, we are interested in odd primes $p$ such that $p \nmid a_n$ for all $n\in \mathbb{N}$, since such primes are not divisors of third components of triples in $P_m$. Prime factors and non-factors of terms of recursions of degree 2 have been studied in \cite{ward}. Every prime $p$ divides $F_n$ for some $n>0$. The least such $n$ is called the rank of $p$ 
and is denoted by $\rho=\rho(p)$, and $p \mid F_n$ if and only if $\rho \mid n$. 

\begin{lemma}\label{rankprime}
$\rho(p)$ is odd if and only if $p \nmid a_n$ for all $n\in \mathbb{N}$. 
\end{lemma}

\begin{proof}
If $\rho$ is even, then identity (iii) of Proposition \ref{properties}
with $n=\rho/2$ implies that $p \mid F_nG_n$, hence $p\mid G_n$, since $p \nmid F_n$. 
The converse follows similarly (note that $p \nmid \gcd(F_n,G_n)$ by identity (iv) of Proposition \ref{properties}). 
\end{proof}

\subsection{The polynomial $f(x)=x^8-2ax^4+1$}

By replacing $x$ by $x^4$ in \eqref{equad}, we obtain the polynomial $f(x) = x^8 - 2ax^4  + 1$ which plays an important role in the next section. Here, we determine when $f(x)$ is reducible, and we show that its splitting field over the rationals does not contain the splitting field of $x^8+1$. Let 
$\alpha=\sqrt[4]{a + b\sqrt{m}}$ be a real root of $f(x)$. Then clearly:
$$
f(x) = (x \pm \alpha)(x \pm 1/\alpha)(x \pm i\alpha)(x \pm i/\alpha).
$$

\begin{prop}\label{polynom}
$2(a+1)$ is never a perfect square; moreover, the following properties are equivalent:
\begin{itemize}
\item[i)] $2(a-1)$ is a perfect square.
\item[ii)] $\O_{\rati[\sqrt{m}]}$ contains units of norm -1.
\item[iii)] The negative Pell's equation $X^2-mY^2=-1$ has integer solutions. 
\item[iv)] $f(x)$ is reducible over $\mathbb{Q}$.
\item[v)] $f(x)= (x^4 -\sqrt{2(a-1)}x^2 - 1)(x^4 +\sqrt{2(a-1)}x^2 -1) : = P_-(x)P_+(x),
$
and $P_{-}(x),P_+(x) \in \mathbb{Q}[x]$ are both irreducible. 
\end{itemize}
\end{prop}

\begin{proof}
If $2(a+1)=t^2$ for $t\in \mathbb{Z}$, then $t$ is even and it follows from $a^2-mb^2=1$ that $mb^2=t^4/4-t^2=t^2((t/2)^2-1)$. Since $m$ is square-free, one must have $t \mid b$, and so $(t/2)^2-m(b/t)^2=1$. This contradicts the assumption that $(a,b)$ is the least positive solution of \eqref{pelleq}. Similarly, if $2(a-1)=t^2$ for $t\in \mathbb{Z}$, one must have $(t/2)^2  + 1 = m(b/t)^2$, that is the 
negative Pell's equation has integer solutions, which is equivalent to $\O_{\mathbb{Q}[\sqrt{m}]}$ 
having units of norm -1. 

If $2(a-1)$ is a perfect square, then it is straightforward to check that $f(x)=P_{-}(x)P_{+}(x)$ as in (v). Since $P_{-}(x)=(x\pm \alpha)(x\pm i/\alpha)$ and $P_{+}(x)=(x\pm 1/\alpha)(x\pm i \alpha)$, and $\alpha^2 \notin \mathbb{Q}$, both are irreducible over $\mathbb{Q}$.

It is left to show that if $f(x)$ is reducible then $2(a-1)$ is a perfect square. Clearly $f(x)$ cannot have a degree-1 factor in $\mathbb{Q}[x]$ since none of its roots are rational. It follows that $f(x)$ cannot have degree-3 or degree-5 factors either. Suppose $f(x)$ has a degree-2 factor, and so $g(x)=(x-\gamma_1)(x-\gamma_2)$ where $\gamma_1,\gamma_2$ are two roots of $f(x)$. It is easy to check that $\{\gamma_1,\gamma_2\}=\{\pm \alpha, \pm 1/\alpha\}$ is the only nontrivial case to consider. In this case, one must have $\gamma_1+\gamma_2 \in \mathbb{Q}$, which gives $\alpha \pm 1/\alpha \in \mathbb{Q}$. But note that then
$$
2(a+1) = \alpha ^4 + 2 + 1/\alpha^4 = ((\alpha \pm 1/\alpha)^2 \mp 2)^2,
$$
which contradicts the fact that $2(a+1)$ is never a perfect square. Finally suppose $f(x)$ has a degree-4 factor $g(x)=(x-\gamma_1)(x-\gamma_2)(x-\gamma_3)(x-\gamma_4)$ in $\mathbb{Q}[x]$. A similar analysis shows that the only nontrivial case to consider is when $\{\gamma_1,\gamma_2,\gamma_3, \gamma_4\}=\{\mu \alpha, - \mu \alpha, \eta/\alpha, -\eta /\alpha\}$, where $\mu, \eta \in \{1,i\}$. In both cases one has $2\alpha^2 \pm 2/\alpha^2=\gamma_1^2+\gamma_2^2+\gamma_3^2+\gamma_4^2 \in \mathbb{Q}$. It follows that
$$2(a \pm 1)=(\alpha^2 \pm 1/\alpha^2)^2$$
is a perfect square. Since $2(a+1)$ is not perfect square, we conclude that $2(a-1)$ is a perfect square, and the proof is completed. 
\end{proof}

\begin{definition}
Let $K$ be the splitting field of $f(x)$, or if $f(x)$ is reducible, let $K$ be the splitting field of $P_-(x)$. Also let $L=\rati[\zeta_{16}]$, where $\zeta_{16}$ is the primitive $16^{th}$ root of unity.  
\end{definition}

Next three lemmas imply that $L \nsubseteq K$.

\begin{lemma} \label{reduciblecase}
Suppose that $f(x)=x^8-2ax^4+1$ is reducible. Then $L \nsubseteq K$. 
\end{lemma}
\begin{proof}
In this case, $K$ is the splitting field of $P_-(x)=x^4 - \sqrt{2(a-1)}x^2 - 1 = (x\pm \alpha)(x \pm i/\alpha)$, which is irreducible in $\rati[x]$ by Proposition \ref{polynom}. A simple application of \cite[Proposition 7.147]{Bordelles} shows that the Galois group $Gal(K/\rati)\cong D_4$. Since the Galois group $Gal(\rati[\zeta_{16}]/\rati)$ is isomorphic to the product of two cyclic groups of order 2 and 4 (\cite[Proposition 7.68]{Bordelles}), we have $L \nsubseteq K$ by the fundamental theorem of Galois theory. 
\end{proof}

If $f(x)$ is irreducible, then its splitting field $K$ is a linear vector space over $\mathbb{Q}[\sqrt{m}]$ with the basis
$$
{\mathcal B}=\{1, \alpha, \alpha^2, \alpha^3, i,  i \alpha, i \alpha^2, i\alpha^3\}.
$$
To see this, it is sufficient to show that $1,\alpha, \alpha^2$, and $\alpha^3$ are linearly independent over $\mathbb{Q}[\sqrt{m}]$. First note that if $\alpha^3=x+y\alpha+z\alpha^2$ for some $x,y,z \in \mathbb{Q}[\sqrt{m}]$, then by multiplying by $\alpha$, we must have $1,\alpha$, and $\alpha^2$ linearly dependent over $\mathbb{Q}[\sqrt{m}]$. However, $1,\alpha, \alpha^2$ are linearly dependent if and only if the 
ring of integers of $\mathbb{Q}[\sqrt{m}]$ has units of norm -1, which is disallowed by Proposition \ref{polynom}. 

\begin{lemma}\label{pm1}
Suppose that $f(x)$ is irreducible in $\mathbb{Q}[x]$. Then $\sqrt{2} \in K$ if and only if $a+1$ or $a-1$ is a perfect square, in which case $\sqrt{2}=h \alpha^2$ for some $h\in \mathbb{Q}[\sqrt{m}]$. 
 \end{lemma}
\begin{proof}
Suppose on the contrary that there exist $x,y,z,t \in \mathbb{Q}[\sqrt{m}]$ such that 
$$\sqrt{2}=x+y\alpha+z \alpha^2 + t \alpha^3.$$
It follows that
$$2=x^2+z^2 \alpha^4+2yt \alpha^4 +(2xy+2zt\alpha^4)\alpha + (y^2+t^2\alpha^4+2xz) \alpha^2 + (2xt+2yz) \alpha^3.$$
It follows that:
\begin{eqnarray} \label{equ1}
x^2+z^2\alpha^4+2yt\alpha^4 &=& 2; \\ \label{equ2}
2xy+2zt\alpha^4 &=& 0;\\ \label{equ3}
y^2+t^2 \alpha^4 +2xz &=& 0;\\ \label{equ4}
2xt+2yz &=&0.
\end{eqnarray}
There are two cases:

Case 1: $z=0$. Equation \eqref{equ3} states that $y^2+t^2 \alpha^4=0$, hence $y=t=0$. Equation \eqref{equ1} then gives $x^2=2$, but $\sqrt{2} \notin \mathbb{Q}[\sqrt{m}]$ for $m>2$, a contradiction. Recall that if $m=2$, then $a=3$ and $f(x)$ is reducible.

Case 2: $z\neq 0$. From equations \eqref{equ2} and \eqref{equ4}, it follows that 
$xyt=-zt^2\alpha^4$ and $xty=-y^2z$, and so $y^2z=zt^2 \alpha^4$. It follows that 
$y^2=t^2 \alpha^4$, and so $y=\pm t\alpha^2$. Since $\{1,\alpha^2\}$ is linearly independent over 
$\mathbb{Q}[\sqrt{m}]$, we must have $t = y = 0$, which implies that $xz=0$ and so $x=0$. 
It then follows from \eqref{equ1} that $z^2\alpha^4=2$. Therefore, $z^2=2(a-b\sqrt{m})$. 
Let $z=s+t\sqrt{m}$ where $s,t \in \mathbb{Q}$. We have 
$$s^2+mt^2=2a,~st=-b \Rightarrow mt^4-2at^2+b^2=0 \Rightarrow t^2=\dfrac{a\pm 1}{m}.$$
Since $(a+1)(a-1)/m=b^2$, it follows that either $a+1$ or $a-1$ must be a perfect square. 
In this case $\sqrt{2}=h \alpha^2$, where $h=\pm z \in \mathbb{Q}[\sqrt{m}]$. 
\end{proof}

\begin{lemma}\label{pm2}
Suppose that $f(x)$ is irreducible in $\mathbb{Q}[x]$. Suppose also that $a+1$ or $a-1$ is a perfect square. Then $\sqrt{2+\sqrt{2}} \notin K$. 
\end{lemma}
\begin{proof}
We have $\sqrt 2=h\alpha^2$ for some $h \in \mathbb{Q}[\sqrt{m}]$ by Lemma \ref{pm1}. On the contrary, suppose 
$$\sqrt{2+\sqrt{2}}=x+y\alpha+z \alpha^2 + t \alpha^3,$$
where $x,y,z,t \in \mathbb{Q}[\sqrt{m}]$. After squaring and regrouping, we must have
\begin{eqnarray}\label{equa1}
x^2+z^2 \alpha^4+2yt \alpha^4 &=& 2;\\ \label{equa2}
2xy+2zt\alpha^4 &=&0;\\ \label{equa3}
y^2+t^2\alpha^4+2xz  &=& h;\\ \label{equa4}
2xt+2yz &=&0.
\end{eqnarray}
There are two cases:

Case 1: $z=0$. It follows from equations \eqref{equa2} and \eqref{equa4} that eiher $x=0$ or $y=t=0$. The latter contradicts \eqref{equa3}, hence $x=0$. Then equations \eqref{equa1} and \eqref{equa3} turn into:
$$yt\alpha^4=1,~y^2+t^2\alpha^4=h.$$
Eliminating $y$ yields $t^4-h\alpha^{-4}t^2+\alpha^{-12}=0$ which cannot have real solutions since $\Delta=h^2 \alpha^{-8}-4\alpha^{-12}=\alpha^{-8}(h^2-4\alpha^{-4})=-2\alpha^{-12}<0$, since $h=\sqrt{2}\alpha^{-2}$. 

Case 2: $z\neq 0$. Equations \eqref{equa2} and \eqref{equa4} imply that $y=\pm t \alpha^2$, but then $y=t=0$ and equations \eqref{equa1} and \eqref{equa3} turn into:
$$2xz=h,~x^2+z^2\alpha^4=2.$$
Eliminating $x$ yields $2z^4-4\alpha^{-4}z^2+\alpha^{-8}=0$. The solutions of this equation are 
$z^2=(2 \pm \sqrt{2})/(2\alpha^{4}).$ It follows that 
$2\pm \sqrt{2} = 2z^2\alpha^4\in \mathbb{Q}[\sqrt{m}]$ which is impossible for $m>2$. 
\end{proof}

\begin{cor}\label{corus}
$L \nsubseteq K$. 
\end{cor}

\begin{proof}
If $f(x)$ is reducible, then the claim follows from Lemma \ref{reduciblecase}. If $f(x)$ is irreducible, Lemmas \ref{pm1} and \ref{pm2} show that $\sqrt{2+\sqrt{2}} \notin K$. However the real part of $2\zeta_{16}$ is $\sqrt{2+\sqrt{2}}$, and so $L \nsubseteq K$ in this case as well. 
\end{proof}


\section{Rank of the quotient group $G_m/P_m$}

Our goal in this section is to show that the quotient group $G_m/P_m$ has infinite rank for all square-free $m>1$. We begin with the following definition.

\begin{definition}
Let $\Lambda_m$ be the set of odd rational primes $p$ for which $(\frac{m}{p})=(\frac{-m}{p})=1$ and 
$1 + my^2 = (pz)^2$ has no integer solutions.
\end{definition}

Recall that the condition that $1+my^2=(pz)^2$ has no integer solutions is equivalent to $p \nmid a_n$ for all $n\in \mathbb{N}$.

\begin{lemma}\label{conda2}
If $p>a^2$ is a prime such that $p \neq 1 \pmod{16}$ and $f(x)=x^8-2ax^4+1$ splits completely modulo $p$, then $p \in \Lambda_m$.
\end{lemma}

\begin{proof}
If $x$ is a root of $f(x)$ modulo $p$, then so is $-x$ and $1/x$. Therefore, there exist $u,v \in \mathbb{Z}_p$ such that the roots of $f(x)$ modulo $p$ are $\pm u, \pm v, \pm u^{-1}, \pm v^{-1}$. It follows that 
$$f(x)=(x\pm u)(x \pm v)(x \pm u^{-1})(x \pm v^{-1}),$$
and so $u^2+v^2+u^{-2}+v^{-2}= 0 \pmod p$. Therefore, $u^2(1+(u^{-1}v)^2)(1+(uv)^{-2})=0 \pmod p$. We conclude that either $(u^{-1}v)^2=-1 \pmod p$ or $(uv)^2=-1 \pmod p$. In either case we must have $p=1 \pmod 4$. Moreover, 
$$a^2-1=u^8-2au^4+a^2=(u^4-a)^2,$$
which implies that 
$$\left (\frac{-m}{p} \right )=\left (\frac{-1}{p} \right ) \left (\frac{a^2-1}{p} \right )=1,$$
since $-mb^2=-(1-a^2)$.

To show that $1 + my^2 = (pz)^2$ has no integer solutions, we need to show that $p \nmid a_n$ for all $n\in \mathbb{N}$. Let $r=u^4,s=u^{-4}$. We note that $r,s$ are solutions of equation \eqref{equad} and $rs=1$. Moreover, $r \neq s \pmod p$; otherwise, $u^8=r^2=rs=1 \pmod p$, while $0=f(u)=u^8-2au^4+1=2-2au^4 \pmod p$, which would imply that $u^{-4}=a \pmod p$. It would follow that $1=u^{-8}=a^2 \pmod p$, which would contradict the assumption that $p>a^2$. Hence $r\neq s \pmod p$. Now let $\rho=\rho(p)$ be the rank of $p$. Since $p=1 \pmod 4$ and $p \neq 1 \pmod{16}$, then $p=8k+5$ or $p=16k+9$ for some integer $k$. From identity (i) of Proposition \ref{properties}, one has
 $$(r-s)F_{2k+1}=r^{2k+1}-s^{2k+1}=s^{2k+1}(u^{16k+8}-1)=0 \pmod p,$$
which implies that $p \mid F_{2k+1}$. It follows that $\rho  \mid (2k+1)$, hence $p \nmid a_n$ for all $n\in \natu$ by Lemma \ref{rankprime}. It follows that $p \in \Lambda_m$ by definition. 
\end{proof}

One can use the Frobenius Density Theorem to show that $\Lambda_m$ is infinite. More precisely, we have the following.

\begin{prop}\label{cor55}
There exist infinitely many primes $p\neq 1 \pmod{16}$ such that $f(x)$ splits completely modulo $p$. 
\end{prop}

\begin{proof}
Let $K$ and $L$ be as in Corollary \ref{corus}. Then $K$ and $L$ are normal extensions of $\mathbb{Q}$ and therefore, by the Inclusion Theorem (see \cite{Wyman} or Corollary 5.5 of \cite{Janusz}), there exist infinitely many primes that completely split in $K$ but not in $L$. It follows that there exist infinitely many primes $p$ such that $f(x)$ factors into linear factors in $\mathbb{Z}_p$ but $x^8+1$ does not. The statement follows from the Cyclotomic Reciprocity Law, which states that $x^8+1$ splits completely modulo an odd prime $p$ if and only if $p=1 \pmod{16}$. 
\end{proof}

\begin{prop}\label{quotientrank}
$G_m / P_m$ has infinite rank for all $m>1$.
\end{prop}
\begin{proof}
Lemma \ref{conda2} and Proposition \ref{cor55} show that $|\Lambda_m|=\infty$. 
Using the ideal class number $h$ of the field $\rati[\sqrt{-m}]$, for each 
$p\in\Lambda_m$, one obtains positive integers $k$, $x$, and $y$ such that $\gcd(x,y)=1$, and 
$x^2 + my^2 = (2^{\delta}p^k)^2$ with $\delta =  0 ~ \mbox{or} ~ 1$. To be more precise, if 
$-m = 5 \pmod 8$, Theorem 12.7.1 of \cite{Alaca} 
provides coprime integers $u$ and $v$ such that 
\begin{equation}\label{ph}
p^h = u^2 + uv +\frac{1}{4}(m+1)v^2 = u^2 + uv + \left(\frac{v}{2}\right)^2 + m\left(\frac{v}{2}\right)^2,
\end{equation}
and hence, if $h = 2t$ we have either $
p^{2t} = (u + v/2)^2 + m(v/2)^2$ or $(2p^t)^2 = (2u + v)^2 + mv^2,
$
depending on whether $v$ is even or odd respectively. If $h$ is odd, squaring equation \eqref{ph} yields:
$$
p^{2h} = \left( (u + v/2)^2 - m\left(\frac{v}{2}\right)^2\right)^2 + m \left(uv + \frac{v^2}{2}\right)^2,
$$
and if $v$ is odd, then $(2p^h)^2 = ( (2u+v)^2-mv^2)^2 + m(2uv + v^2)^2.
$
Note that it is possible to have $k< h/2$ when $h$ is even, or $k < h$ when $h$ is odd because of possible common factors among the components of the triples obtained above. For the other congruence classes of $m$, Theorem 12.7.1 gives coprime integers $u$ and $v$ such that 
$p^h = u^2 + mv^2$, and similarly we get an element of $G_m$ by squaring this equation if $h$ is odd. 

So far, for each prime $p\in \Lambda_m$, we have a primitive triple $[x_p,y_p,z_p]$ such that the only odd prime factor of $z_p$ is $p$. We claim that the cosets $P_m+[x_p,y_p,z_p]$ in $G_m/P_m$ are linearly independent over $\mathbb{Z}$. On the contrary, suppose there exists a finite set $\Omega \subseteq \Lambda_m$ and nonzero integers $k_p$, $p \in \Omega$, such that 
$$\sum_{p\in \Omega} k_p[x_p,y_p,z_p]  \in P_m.$$
Note that the third component of the primitive representation of $k_p[x_p,y_p,z_p]$ is divisible by $p$, since otherwise $[x_p,y_p,z_p]$ would have to be an element of finite order which is impossible except if $[x_p,y_p,z_p]=[1,1,2] \in P_3$; however, this contradicts $p \mid z_p$. 

Let $[x_0,y_0,z_0]$ be the primitive representation of $\sum_{p\in \Omega} k_p[x_p,y_p,z_p]$. Since 2 is the only possible common divisor of $z_{p_1}$ and $z_{p_2}$ for distinct $p_1,~p_2\in \Omega$, Lemma \ref{primitive} implies that $z_0$ is divisible by $p$ for all $p\in \Omega$. Since $[x_0,y_0,z_0] \in P_m$, there exist triples $[1,y^\prime_i, x^\prime_i] \in G_m$ and nonzero integers $l_i$, $i=1,\ldots, k$, such that the primitive representation of $\sum_{i=1}^k l_i[1,y^\prime_i,x^\prime_i]$ is $[x_0,y_0,z_0]$. Since $p \mid z_0$, we must have $p \mid x^\prime_i$ for some $1\leq i \leq k$. But then the equation $1+my^2=(pz)^2$ has the solution $y=y_i^\prime$ and $z=x_i^\prime/p$, which contradicts $p\in \Lambda_m$. Therefore, the cosets $P_m+[x_p,y_p,z_p]$, $p\in \Lambda_m$, are linearly independent over $\mathbb{Z}$. Since $|\Lambda_m|=\infty$, we conclude that the rank of $G_m/P_m$ is infinite. 
\end{proof}


\section{Examples of nontrivial torsion}

In this section, we show examples of $m$ for which $G_m/P_m$ has nontrivial torsion. We construct a homomorphism $f_m:G_m\to Cl(F_m)$, 
where $Cl(F_m)$ stands for the ideal class group of the imaginary quadratic field $F_m=\rati[\sqrt{-m}]$. 

\begin{definition}
Let $[a,b,c]$ represent a primitive triple in $G_m$. Let
$$
f_m([a,b,c])=\left [ \left \langle c,~a+b\sqrt{-m} \right \rangle \right ],
$$
where the integral ideal $\left \le c,~a+b\sqrt{-m} \right \re$ is the set of all linear combinations of $c$ and $a+b\sqrt{-m}$ with coefficients in $\O_{F_m}$ 
(the ring of integers in $F_m$), and
$\left [ \left \le c,~a+b\sqrt{-m} \right \re \right ]$ is the corresponding ideal class in the ideal class group $Cl(F_m)$.
\end{definition}

\noindent Notice that $f_m$ is well-defined, since $\left \langle c,~a+b\sqrt{-m} \right \rangle = \left  \langle c, -a-b\sqrt{-m}\right \rangle$. 

\begin{prop}
The map
$
f_m:G_m\to Cl(F_m)$ is a group homomorphism.
\end{prop}
\begin{proof}
Suppose that $[a,b,c]+[x,y,z]=[\alpha,\beta,\gamma]$, 
where $\gcd(\alpha,\beta,\gamma)=1$. Since 
$$\left [ \left \le cz,~(ax-mby) + (ay+bx)\sqrt{-m}\right \re \right ] = \left [ \left \le \gamma,~\alpha+\beta\sqrt{-m}\right \re \right ],$$
the statement follows if we show that the product of ideals 
\begin{equation}
\label{product}
\left \le c,~a+b\sqrt{-m}\right \re \left  \le z,~x+y\sqrt{-m} \right \re \left \le cz,~(a-b\sqrt{-m})(x-y\sqrt{-m}) \right \re
\end{equation}
is a principal ideal of $\O_{F_m}$. The ideal $\left \le cz,~(a-b\sqrt{-m})(x-y\sqrt{-m})\right \re$ is the conjugate of  
$\left \le cz,~(ax-mby) + (ay+bx)\sqrt{-m} \right \re$, hence its class is the inverse of $\left [ \left \le \gamma,~\alpha+\beta\sqrt{-m} \right \re \right ]$ in the ideal class group. A straightforward 
computation shows that the product in \eqref{product} can be written as $\le cz\re I$, where the ideal $I$ is  
\begin{equation}\nonumber
\left \le cz, 2za, 2cx, 2(ax-mby),z(a+b\sqrt{-m}), c(x+y\sqrt{-m}), (ax-mby)+(ay+bx)\sqrt{-m} \right  \re,
\end{equation}
as linear combinations with coefficients in $\O_{F_m}$. Since $\gcd(c,a)=\gcd(z,x)=1$, we have $2z= (cz)2n_1 + (2az)n_2$ for some $n_1,n_2\in\inte$, which allows us to replace $2za$ by $2z$ in the presentation of $I$. Similarly, we can replace $2cx$ by $2c$. Let $d= \gcd(c,z,ax-mby)$ so that for some integers $N_1, N_2,$ and $N_3$, we have
\begin{equation}
\label{combination}
2d = 2cN_1 + 2zN_2 +2(ax-mby)N_3.
\end{equation}
Since $(ax-mby)^2 + m(ay+bx)^2 = (cz)^2$ and $\gcd(a,c)=1$, one has $d \mid (ay+bx)$, and so $I=\le d \re I_1$, where \begin{equation}
\label{I2}
I_1= \left \le \frac{cz}{d},2,\frac{z}{d}(a+b\sqrt{-m}),\frac{c}{d}(x+y\sqrt{-m}),\frac{ax-mby}{d} + \frac{ay+bx}{d}\sqrt{-m} \right \re.
\end{equation}
If $m= 1,2\pmod{4}$, then both $c$ and $z$ are odd. Therefore, 1 is a linear combination of $cz/d$ and $2$, which implies that $I=\le d\re $ and the proof is completed in this case. Thus, suppose that $m= 3 \pmod{4}$. By subtracting certain multiples of 2 from the generators of $I_1$, we can rewrite $I_1$ in one of the following forms.
$$
I_1 =
\begin{cases}
\le 2 \re \\
\le 2,~1\re  = \O_{F_m}\\
\le 2,~\sqrt{-m}\re  = \O_{F_m}\\
\le 2,~1+\sqrt{-m} \re =  \le 2\re \le 1,~\frac{1+\sqrt{-m}}{2}\re=\le 2\re \\
\end{cases}
$$
In each of these cases, $I$ is a principal ideal and the proposition follows.
\end{proof}

\begin{remark}
Since $\le c\re \left \le c,~a+b\sqrt{-m} \right \re = \left \le a+b\sqrt{-m} \right \re \left \le c,~a-b\sqrt{-m} \right \re$, it follows that 
$f_m([a,b,c])=f_m([a,-b,c])$ and the image of $f_m$ consists of elements of order 2 (these are the so-called ambiguous classes). 
In particular, the homomorphism $f_m$ in general is not onto.
\end{remark}

The next lemma shows that $f_m$ gives an invariant for the elements of the quotient group $G_m/P_m$ for certain values of $m$.

\begin{lemma}\label{homom}
Let $m$ be a positive square-free integer such that the ring of integers of the real quadratic 
field $\rati[\sqrt{m}]$ contains units of norm $-1$. Then 
$$
P_m\subseteq \ker(f_m).
$$
\end{lemma}
\begin{proof}
Let $(a_0,b_0)$ be the least positive solution of 
\begin{equation}\label{pell2}
x^2 - my^2 = -1.
\end{equation}
 According to theorems 7.25 and 7.26 of \cite{Niven}, for every generator $[1,y,z]$ of $P_m$, one has $$
z+y\sqrt{m} = a_{2n} + b_{2n}\sqrt{m} = (a_0+b_0\sqrt{m})^{2n},$$
for some $n\in \natu$. Therefore, $z = a_n^2+mb_n^2$, $y = 2a_nb_n$. It follows from $a_n^2 - mb_n^2=\pm 1$ that
\begin{eqnarray}\nonumber
\left \le z,1+y\sqrt{-m} \right \re &=& \left  \le a_n^2+mb_n^2,1+2a_nb_n\sqrt{-m} \right \re \\ \nonumber
& = & \left  \le (a_n+b_n\sqrt{-m})(a_n-b_n\sqrt{-m}),(a_n \pm b_n \sqrt{-m})^2 \right \re \\ \nonumber
&=& \le a_n\pm b_n\sqrt{-m}\re \le 2a_n,a_n + b_n\sqrt{-m}\re  \\ \nonumber
&=& \le a_n\pm b_n\sqrt{-m}\re  \O_{F_m},
\end{eqnarray}
where the last equality follows from 
$$
(a_n + b_n\sqrt{-m})^2- (2a_n)(b_n\sqrt{-m}) = a_n^2-mb_n^2 = \pm 1.
$$
Therefore, $f_m$ maps every generator of $P_m$ to a principal ideal of $\O_{F_m}$, and the lemma follows.
\end{proof}

 Existence of units of norm $-1$ in Lemma \ref{homom} is crucial as shown by the next  example.

\begin{example}
Let $m=6$ and $[1,2,5]\in P_6$. The group $Cl(F_6)\cong \inte/2\inte$ is generated by the class of ideal $\le 2,~\sqrt{-6}\re$ (see \cite[Section 12.6]{Alaca}). Since $f_6([1,2,5]) = [\le 5,~1+2\sqrt{-6}\re] = [\le 5,~2 - \sqrt{-6}\re]$ and 
$$
\le 2\re\le 5,2 - \sqrt{-6}\re = \le 2- \sqrt{-6}\re\le 2,2 + \sqrt{-6}\re = \le 2-\sqrt{-6}\re \le 2,\sqrt{-6}\re,
$$
we see that $[1,2,5]\notin \ker(f_6)$.
\end{example}

Next, we show that $G_5/P_5$ has nontrivial 2-torsion.

\begin{example}\label{exgen}
For $m=5$, the least solution of \eqref{pell2} is given by $(a_0,b_0) = (2,1)$. We have $[2,1,3]\in G_5$ and $[2,1,3] + [2,1,3] = [-1,4,9]\in P_5$. On the other hand, $f_5([2,1,3]) = [\le 3,2+\sqrt{-5}\re]$, $\le 3,2+\sqrt{-5}\re = \le 3,1-\sqrt{-5}\re$ and 
$\le 2\re\le  3,1-\sqrt{-5}\re = \le 1-\sqrt{-5}\re\le 2,1+\sqrt{-5}\re$ we have $f_5([2,1,3]) = [\le 2,1+\sqrt{-5}\re]$.
The ideal class $[\le 2,1+\sqrt{-5}\re]$ is a generator of the ideal class group of $\rati[\sqrt{-5}]$ (see section \cite[Section 12.6]{Alaca}) and hence by the lemma above $[2,1,3]\notin P_5$. It follows that the class $[2,1,3] +P_5\in G_5/P_5$ is a nontrivial element of order 2.
\end{example}

Example \ref{exgen} can be generalized as follows.

\begin{lemma}\label{np}
Suppose there exists a primitive triple $[a,b,c]\in G_m$ such that $m>c$ and $a^2 - mb^2 = -1$. 
Then the ideal $\le c,~a+b\sqrt{-m}\re$ is not principal.
\end{lemma}
\begin{proof}
Suppose on the contrary that there exists $z\in\O_{F_m}$ such that $ \le c,~a+b\sqrt{-m}\re = \le z\re$. It follows from the assumptions that $c^2 = 2a^2+1$, and so $a$ must be even and $m= 1\pmod{4}$. Therefore, there exist $x,y\in\inte$ such that $z=x+y\sqrt{-m}$. Choose $\gamma,\beta\in\inte$ such that $c\gamma+b\beta=1$. Then $
\le c,a + b\sqrt{-m}\re = \le c,a\beta + \sqrt{-m}\re$ with $b^2(a^2\beta^2+m) = c L$ for some $L\in\inte$. 
It means that the ideal $\le c,~a + b\sqrt{-m}\re$ has $\{c,~a\beta + \sqrt{-m}\}$ as an integral basis (see\cite[Section 6.5]{Alaca}). Using the norm, we see that
$$x^2+my^2=N\left ( \left \le x+y\sqrt{-m} \right \re \right )=N \left ( \left \le c,~a + b\sqrt{-m} \right \re \right ) = c,$$ 
by \cite[Theorem 7.1.5]{Alaca}. This contradicts the given condition $m> c$ (primitivity of $[a,b,c]$ implies $y\neq 0$).
\end{proof}

As a corollary to Lemmas \ref{homom} and \ref{np}, in Table \ref{table2}, several examples of $m$ and primitive triples $[a,b,c]+P_m\in G_m/P_m$ of order two are given.
{
\renewcommand{\arraystretch}{2}
\begin{table}
\begin{center}
\begin{tabular}{||c|c|c|c||}
\hline
m & a & b & c \\
\hline
5 & 2 & 1 & 3 \\
\hline
145 & 12 & 1 & 17 \\
\hline
985 & 408 & 13 & 577 \\
\hline
5654885 & 2378 &1 & 3363\\
\hline
192099601 &13860 &1 &19601 \\
\hline
261029261 & 80782 &5 &114243 \\
\hline
\end{tabular}
\end{center}
\caption{}
\label{table2}
\end{table}
}
\begin{remark}\label{spex}
The triple $[70,13,99]\in G_{29}$ dos not satisfy the requirement $m> c$ of Lemma \ref{np};  however, since 
$99 = x^2 + 29y^2$ has no integer solutions, the proof of Lemma \ref{np}. implies that 
$[70,13,99]+P_{29}$ has order two in $G_{29}/P_{29}$. 
\end{remark}

Examples in Table \ref{table2} are found by searching through the decreasing convergents $h_n/k_n$, 
$n\in \natu$, of the continued fraction 
expansion of $\sqrt{2}$. Suppose $a^2-mb^2=-1$ and $a^2+mb^2=c^2$, then $c^2-2a^2=1$. Moreover, $\gcd(c,a)=1$ and $c$ is odd, hence $\gcd(c-a,c+a)=1$. 
It follows from the recursive formulas for the numerators and denominators of the 
convergents $h_n/k_n$ that $c\pm a$ are consecutive denominators of 
the increasing convergents of $\sqrt{2}$ i.e., $c-a=k_{2s-2}$ and $c+a = k_{2s}$ for some $s\in \natu$. 
Therefore, if there are infinitely many square-free denominators $k_{2s}$ of $\sqrt{2}$, 
one concludes that there are infinitely many $m$ for 
which the conditions of Lemma \ref{np} are met, since by letting $m$ equal the square free part of $k_{2s-2} k_{2s}$, we have $m \geq k_{2s} >c$. Hence the infinitude of square-free denominators $k_{2s}$ 
would imply the infinitude of the quotient groups $G_m/P_m$ with nontrivial 2-torsion.

\begin{remark}
Denominators $k_{2s},~s\geq 0$, of increasing convergents of $\sqrt{2}$ satisfy the recurrence 
relation $k_{2(s+1)} = 6k_{2s} - k_{2s-2}$, $s\in\natu$, with $k_0=1,~k_2=5$. It is a difficult problem in general to determine if solutions of second-order linear difference equation have infinitely many square-free terms. For example, the existence of infinitely many square-free Fibonacci 
numbers follows from nonexistence of Wall-Sun-Sun primes, which is an open problem; \cite{Sun}. 
\end{remark}

It is also worth mentioning that the example in Remark \ref{spex} corresponds to 
three consecutive convergents of $\sqrt{2}$ 
$$
\frac{41}{29},~\frac{99}{70},~\frac{239}{169}~~\mbox{with}~~99^2 - 70^2 = 29 13^2 ~~
\mbox{and}~~239^2 - 2 169^2 = -1.
$$
It is well known that the equation $z^2 - 2 y^4 = -1$ has no solutions in positive integers other than 
$(z,y)=(1,1)$ and $(239,13)$; \cite{Lj}. Therefore, it is natural to suspect that the example of 
$[70,13,99]\in G_{29}$ is the only example from the sequence of decreasing convergents of $\sqrt{2}$, 
where the condition $ m > c $ of Lemma \ref{np} is violated.

\end{document}